\documentclass{amsart}
\usepackage{graphicx}
\newtheorem{theorem}{Theorem}[section]
\newtheorem{lemma}[theorem]{Lemma}

\newtheorem{proposition}[theorem]{Proposition}

\theoremstyle{definition}

\newtheorem{example}[theorem]{Example}

\newtheorem{claim}[theorem]{Claim}

\theoremstyle{remark}
\newtheorem{remark}[theorem]{Remark}

\numberwithin{equation}{section}

\begin{document}

\title{The representativity of pretzel knots}

\author{Makoto Ozawa}
\address{Department of Natural Sciences, Faculty of Arts and Sciences, Komazawa University, 1-23-1 Komazawa, Setagaya-ku, Tokyo, 154-8525, Japan}
\curraddr{Department of Mathematics and Statistics, The University of Melbourne
Parkville, Victoria 3010, Australia (temporary)}
\email{w3c@komazawa-u.ac.jp, ozawam@unimelb.edu.au (temporary)}



\dedicatory{In hope of recovery of my father}


\begin{abstract}
In the previous paper, the {\em representativity} of a knot $K$ was defined as
\[
r(K)=\max_{F\in\mathcal{F}} \min_{D\in\mathcal{D}_F} |\partial D\cap K|,
\]
where $\mathcal{F}$ is the set of all closed surfaces of positive genus containing $K$ and $\mathcal{D}_F$ is the set of all compressing disks for $F$ in $S^3$.
Then an inequality $r(K)\le b(K)$ has shown, where $b(K)$ denotes the bridge number of $K$.
This shows that a $(p,q,r)$-pretzel knot has the representativity less than or equal to $3$.

In the present paper, we will show that a $(p,q,r)$-pretzel knot has the representativity $3$ if and only if $(p,q,r)$ is either $\pm(-2,3,3)$ or $\pm(-2,3,5)$.
We also show that a large algebraic knot has the representativity less than or equal to $3$.
\end{abstract}

\maketitle

\section{Introduction}


Let $K$ be a knot in the 3-sphere $S^3$ and $F$ a closed surface of positive genus containing $K$.
In \cite{MO1}, the {\em representativity} $r(F,K)$ of a pair $(F,K)$ is defined as the minimal number of intersecting points of $K$ and $\partial D$, where $D$ ranges over all compressing disks for $F$ in $S^3$.
Furthermore, we defined in \cite{MO} the {\em representativity} $r(K)$ of a knot $K$ as the maximal number of $r(F,K)$ over all closed surfaces $F$ of positive genus containing $K$.
The representativity $r(K)$ of a knot $K$ is a new knot invariant.

The following inequality is the first fundamental result for the representativity.
Let $b(K)$ denote the bridge number of $K$.

\begin{theorem}[{\cite[Theorem 1.2]{MO}}]
For a knot $K$, $r(K)\le b(K)$.
\end{theorem}

This shows that a $(p,q,r)$-pretzel knot has the representativity less than or equal to $3$.
In the present paper, we characterize $(p,q,r)$-pretzel knots with the representativity $3$.

\begin{theorem}\label{main}
Let $K$ be a $(p,q,r)$-pretzel knot in $S^3$.
Then $r(K)=3$ if and only if $(p,q,r)=$ $\pm(-2,3,3)$ or $\pm(-2,3,5)$.
\end{theorem}

\begin{remark}
It was shown in \cite[Theorem III]{AK} that a $(p,q,r)$-pretzel knot is a torus knot if and only if $(p,q,r)=$ $\pm(1,1,1)$, $\pm(-2,3,3)$ or $\pm(-2,3,5)$, and observed that a $\pm(-2,3,3)$-pretzel knot is a $(3,\pm4)$-torus knot and a $\pm(-2,3,5)$-pretzel knot is a $(3,\pm5)$-torus knot.
In the proof of Theorem \ref{main}, we will give unknotted tori containing these pretzel knots directly.
\end{remark}

Let $K$ be a tangle composite knot, that is, which admits at least one essential tangle decomposition.
We define the {\em tangle string number} $ts(K)$ of $K$ as the minimal number of $n$ over all $n$-string essential tangle decompositions of $K$.

\begin{proposition}[{\cite[Example 3.6]{MO}}]\label{ts}
For a tangle composite knot $K$, $r(K)\le 2 ts(K)$.
\end{proposition}

By Proposition \ref{ts}, $r(K)\le 2$ if $K$ is a composite knot, and $r(K)\le 4$ if $K$ has an essential Conway sphere.
Nonetheless, the representativity of a large algebraic knot (i.e. algebraic knot with an essential Conway sphere) has more sharp upper bound.

\begin{theorem}\label{algebraic}
Let $K$ be a large algebraic knot.
Then $r(K)\le 3$.
\end{theorem}



\section{Proofs}

\begin{proof} (of Theorem \ref{main})
Let $K$ be a $(p,q,r)$-pretzel knot with a form of a Montesinos knot of type $(1/p,1/q,1/r)$.
Thus $K$ is obtained from three rational tangles $(B_1,T_1)$, $(B_2,T_2)$ and $(B_3,T_3)$ of slopes $1/p$, $1/q$ and $1/r$ respectively by connecting them in series.
Let $A$ be the ``axis" of the Montesinos knot $K$, that is, an unknotted loop in the complement of $K$ such that there exist three meridian disks $D_1$, $D_2$ and $D_3$ in the solid torus $V=S^3-\text{int}N(A)$ which cut $(V,K)$ into three rational tangles $(B_i,T_i)$, $(B_2,T_2)$ and $(B_3,T_3)$, where $B_1\cap B_2=D_1$, $B_2\cap B_3=D_2$ and $B_3\cap B_1=D_3$.
See Figure \ref{pretzel1}.

\begin{figure}[htbp]
	\begin{center}
	\includegraphics[trim=0mm 0mm 0mm 0mm, width=.6\linewidth]{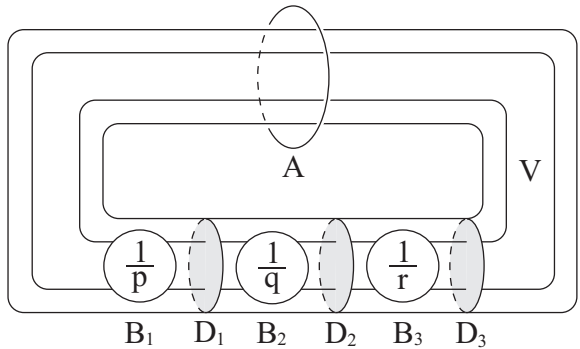}
	\end{center}
	\caption{The ``axis" $A$ of $K$ and its exterior $V$}
	\label{pretzel1}
\end{figure}

Suppose that $r(K)=3$ and let $F$ be a closed surface containing $K$ which satisfies $r(F,K)=3$.
We assume that $|F\cap A|$ is minimal up to isotopy of $F$.
Moreover, we assume that $|F\cap (D_1\cup D_2\cup D_3)|$ is minimal up to isotopy of $F$.
Then, each component of $F\cap \partial V$ is a longitude of $V$ which intersects $\partial D_i$ in one point for each $i$, and $F\cap D_i$ consists of $|F\cap \partial V|/2$-arcs and loops.

\begin{claim}\label{p,q,r}
$|p|,\ |q|,\ |r|\ge 2$.
\end{claim}

\begin{proof}

If $|p|\le 1$, then $K$ is a connected sum of two torus knots, a torus knot, or a 2-bridge knot.
In either cases, we have $r(K)\le 2$, a contradiction.
\end{proof}

In the following, we show that $F\cap D_i$ has no loop and consists of mutually parallel $|F\cap \partial V|/2$-arcs.

\begin{claim}\label{no arc}
There exists no outermost arc $\alpha$ of $F\cap D_i$ with $|\alpha \cap K|=0$.
\end{claim}

\begin{proof}
Let $\alpha$ be an outermost arc of $F\cap D_i$ with $|\alpha \cap K|=0$, and $\delta$ be the corresponding outermost disk.
Since each component of $F\cap \partial V$ intersects $\partial D_i$ in one point for each $i$, $\delta\cap \partial V$ connects two longitudinal components of $F\cap \partial V$.
Hence by an isotopy of $F$ along $\delta$, $|F\cap A|$ can be reduced, a contradiction.
\end{proof}

\begin{claim}\label{no loop}
There exists no innermost loop $\beta$ of $F\cap D_i$ with $|\beta\cap K|\le 1$.
\end{claim}

\begin{proof}
Let $\beta$ be an innermost loop of $F\cap D_i$ with $|\beta\cap K|=0$, and $\delta$ be the corresponding innermost disk.
If $\partial \delta$ is essential in $F$, then $r(F,K)=0$, a contradiction.
Otherwise, $\partial \delta$ bounds a disk $\delta'$ in $F$ and we can remove $\alpha$ from $D_i$ by isotoping $\delta'$ into $\delta$.

Next let $\beta$ be an innermost loop of $F\cap D_i$ with $|\beta\cap K|=1$, and $\delta$ be the corresponding innermost disk.
Since $|\beta\cap K|=1$, $\beta$ is essential in $F$.
Then $r(F,K)\le 1$, a contradiction.
\end{proof}

\begin{claim}\label{no loop2}
There exists no innermost loop $\beta$ of $F\cap D_i$ with $|\beta\cap K|=2$.
\end{claim}

\begin{proof}
Suppose there exists an innermost loop $\beta$ of $F\cap D_i$ with $|\beta\cap K|=2$.
Then $\beta$ bounds a disk $\delta$ in $F$ since $r(K)=3$.
Let $\beta'$ be a loop of $\delta\cap (D_1\cup D_2\cup D_3)$ which is innermost in $\delta$, and $\delta'$ be the corresponding innermost disk in $\delta$.
By Claim \ref{no loop}, $|\beta'\cap K|=2$ and $\delta'$ contains one subarc of $K$.
Since $\delta'$ is parallel into some $D_i$, at least one of $p,\ q,\ r$ is equal to $0$.
This contradicts Claim \ref{p,q,r}.
\end{proof}

By Claims \ref{no arc}, \ref{no loop} and \ref{no loop2}, $F\cap D_i$ has one of the following configurations.
See Figure \ref{conf}.

\begin{itemize}
\item[(1)] $F\cap D_i$ consists of mutually parallel arcs whose two outermost arcs $\alpha_{i1}$, $\alpha_{i2}$ satisfying $|\alpha_{ij}\cap K|=1$ for $j=1,2$.
\item[(2)] $F\cap D_i$ consists of one arc $\alpha_i$ satisfying $|\alpha_i\cap K|=2$.
\end{itemize}

\begin{figure}[htbp]
	\begin{center}
	\begin{tabular}{cc}
	\includegraphics[trim=0mm 0mm 0mm 0mm, width=.2\linewidth]{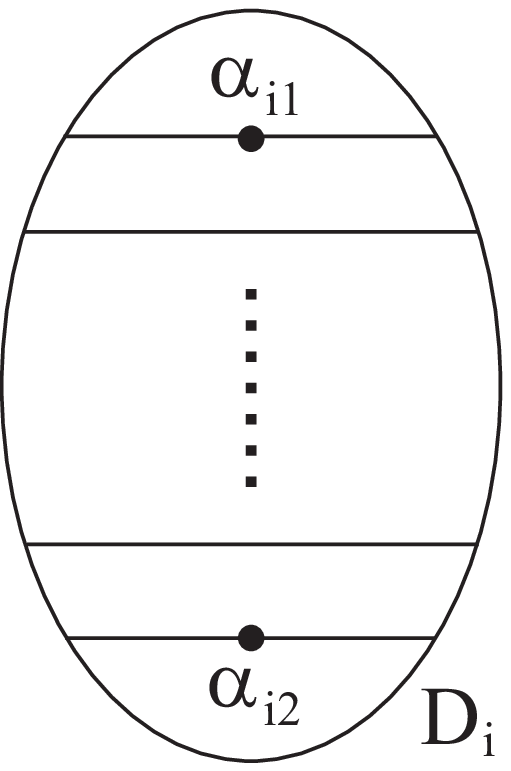}&
	\includegraphics[trim=0mm 0mm 0mm 0mm, width=.2\linewidth]{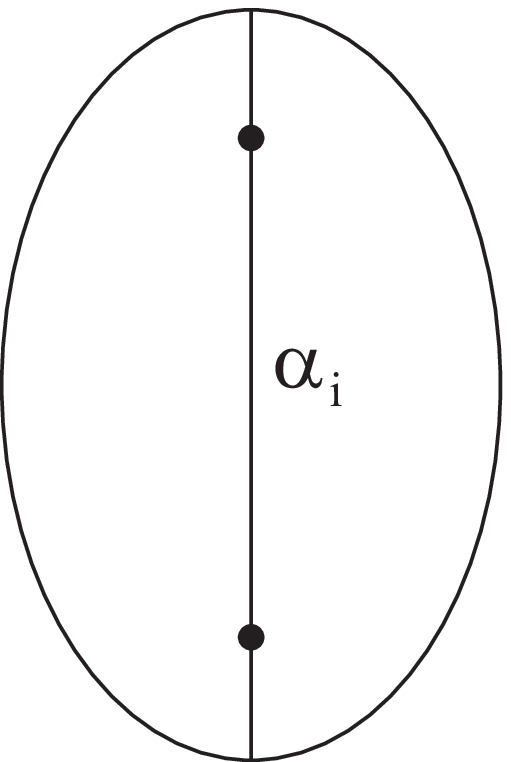}\\
	Configuration (1) & Configuration (2)\\
	\end{tabular}
	\end{center}
	\caption{Two configurations for $F\cap D_i$}
	\label{conf}
\end{figure}

In the following, we show that Configuration (2) does not exist.

\begin{lemma}\label{Morse}
Let $(B,T)$ be a trivial tangle and $F$ is a properly embedded surface in $B$ which contains $T$.
Let $h:B\to \Bbb{R}_{\ge 0}$ be the standard Morse function with one critical point and suppose that each component of $T$ has exactly one maximal point with respect to $h$.
Then a pair $(F,T)$ can be isotoped so that $F$ has no inessential saddle point and each component of $T$ has one maximal point with respect to $h$.
\end{lemma}

\begin{proof}
This lemma essentially follows \cite[Lemma 2.1, 2.2]{MO}.
\end{proof}

\begin{claim}\label{disk}
In Configuration (2), $F\cap B_1$ consists of a disk.
\end{claim}

\begin{proof}
In Configuration (2), $F \cap \partial B_1$ consists of one loop $\beta$ satisfying $|\beta\cap K|=4$.
If $F\cap B_1$ was not a disk, then by Lemma \ref{Morse}, there exists an essential saddle point of $F\cap B_1$ and this shows that $r(F\cap B_1,T_1)\le 2$ as the proof of \cite[Theorem 1.2]{MO}.
\end{proof}

By Claim \ref{disk}, we have $|p|\le 1$.
This contradicts Claim \ref{p,q,r}.
Hence Configuration (2) does not exist.
Hereafter, we assume Configuration (1).
Since each component of $F\cap \partial V$ is a longitude of $V$ which intersects $\partial D_i$ in one point for each $i$, $F\cap \partial B_i$ has either following two cases.
See Figure \ref{case}.

\begin{itemize}
\item [(A)] $F\cap \partial B_i$ consists of mutually parallel loops whose two innermost loops $\beta_{i1},\ \beta_{i2}$ satisfying $|\beta_{ij}\cap K|=2$ for $j=1,2$.
\item [(B)] $F\cap \partial B_i$ consists of one loop $\beta_i$ satisfying $|\beta_i\cap K|=4$.
\end{itemize}

\begin{figure}[htbp]
	\begin{center}
	\begin{tabular}{cc}
	\includegraphics[trim=0mm 0mm 0mm 0mm, width=.3\linewidth]{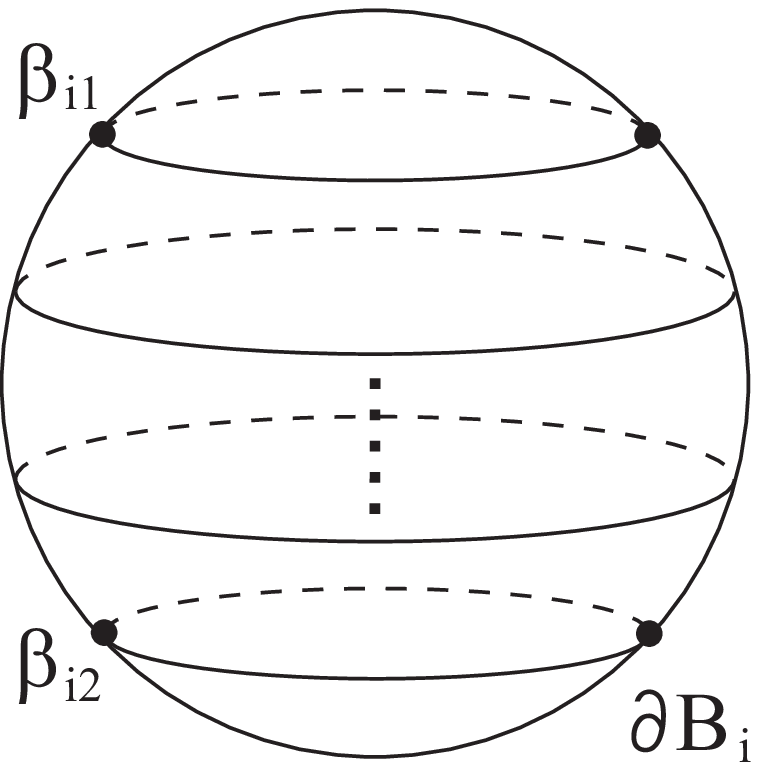}&
	\includegraphics[trim=0mm 0mm 0mm 0mm, width=.3\linewidth]{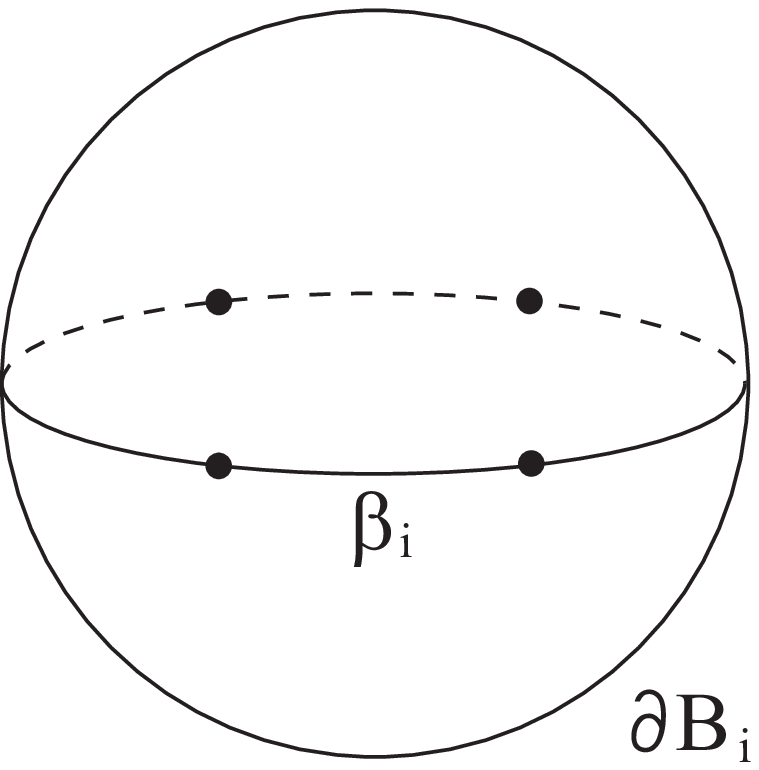}\\
	Case (A) & Case (B)\\
	\end{tabular}
	\end{center}
	\caption{Two cases for $F\cap \partial B_i$}
	\label{case}
\end{figure}

\begin{claim}
In either cases, $F\cap B_i$ consists of disks.
\end{claim}

\begin{proof}
In Case (A), without loss of generality, let $\delta$ be the innermost disk in $\partial B_1$ bounded by $\beta_{11}$.
Since $r(F,K)=3$, $\beta_{11}$ bounds a disk $\delta'$ in $F$.
Let $\gamma$ be a loop of $\delta'\cap \partial B_1$ which is innermost in $\delta'$, and $\epsilon$ be the corresponding innermost disk in $\delta'$.
Note that $\gamma$ is a loop of $F\cap \partial B_1$, and $\epsilon$ contains exactly one subarc of $K$ if and only if $\gamma$ is either $\beta_{11}$ or $\beta_{12}$.

If $\epsilon$ is contained in $B_2\cup B_3$, then an algebraic tangle $(B_2\cup B_3, T_2\cup T_3)$ is inessential.
It follows that $|q|\le 1$ or $|r|\le 1$ and it contradicts Claim \ref{p,q,r}.
Hence $\epsilon$ is contained in $B_1$.

If $\gamma=\beta_{11}$ and it bounds a disk $\epsilon$ in $B_1$, then a loop of $F\cap \partial B_i$ which is next to $\beta_{11}$ on $\partial B_i$, say $\beta'$, also bounds a disk $\epsilon'$ in $B_1$ such that $\epsilon'\cap F=\beta'$.
Since $r(F,K)=3$, $\beta'$ bounds a disk $\epsilon''$ in $F$ and there is an isotopy sending $\epsilon''$ to $\epsilon'$.
By the minimality of $|F\cap A|$ and $|F\cap D_i|$, $\beta'$ bounds a disk $\epsilon'$ of $F\cap B_1$.
Continuing this argument, each loop of $F\cap \partial B_1$ except for $\beta_{12}$ bounds a disk of $F\cap B_1$ in $B_1$.
Now we apply Lemma \ref{Morse} to $(B_1,T_1)$ and $F\cap B_1$.
If $\beta_{12}$ does not bound a disk of $F\cap B_1$, then by Lemma \ref{Morse}, there exists an essential saddle point of $F\cap B_1$ and this shows that $r(K)\le 1$, a contradiction.
Hence each loop of $F\cap \partial B_1$ bounds a disk of $F\cap B_1$ if $\gamma=\beta_{11}$.

If $\gamma\ne \beta_{1j}$, then by the same argument as above, we can show that each loop of $F\cap \partial B_1$ bounds a disk of $F\cap B_1$.

In Case (B), we apply Lemma \ref{Morse} to $(B_i,T_i)$ and $F\cap B_i$.
If $F\cap B_i$ was not a disk, then by Lemma \ref{Morse}, there exists an essential saddle point of $F\cap B_i$ and this shows that $r(F\cap B_1,T_1)\le 2$, a contradiction.
\end{proof}

We call $F\cap B_i$ {\em Type \rm{(A)}} or {\em Type \rm{(B)}} if $F\cap \partial B_i$ has Case (A) or Case (B) respectively.
See Figure \ref{type}.

\begin{itemize}
\item [(A)] $F\cap B_i$ consists of mutually parallel $k$-disks whose two outermost disks $E_{i1}$ and $E_{i2}$ contain one string of $T_i$.
\item [(B)] $F\cap B_i$ consists of one disk which contains two strings of $T_i$.
\end{itemize}

\begin{figure}[htbp]
	\begin{center}
	\begin{tabular}{cc}
	\includegraphics[trim=0mm 0mm 0mm 0mm, width=.3\linewidth]{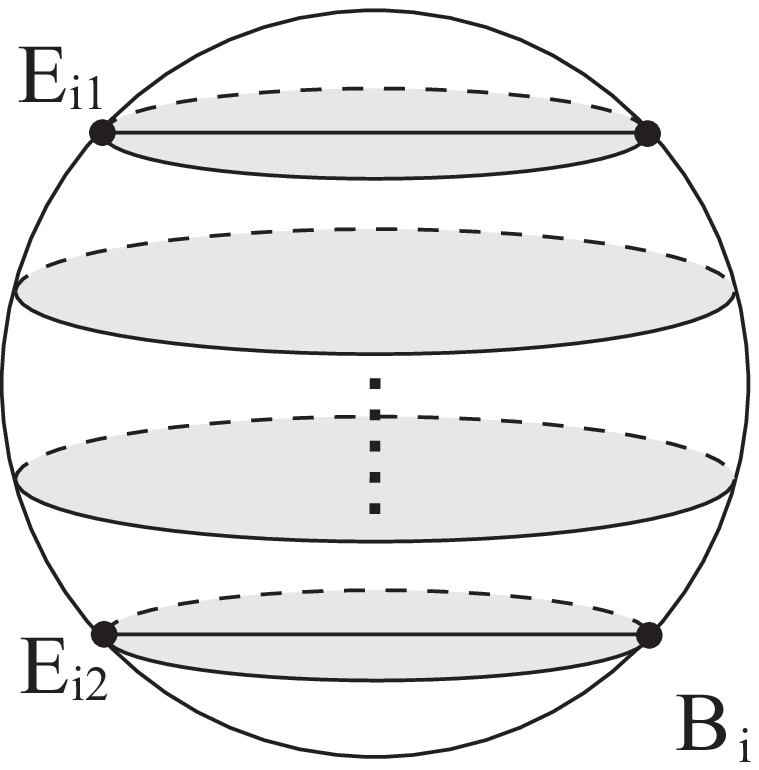}&
	\includegraphics[trim=0mm 0mm 0mm 0mm, width=.3\linewidth]{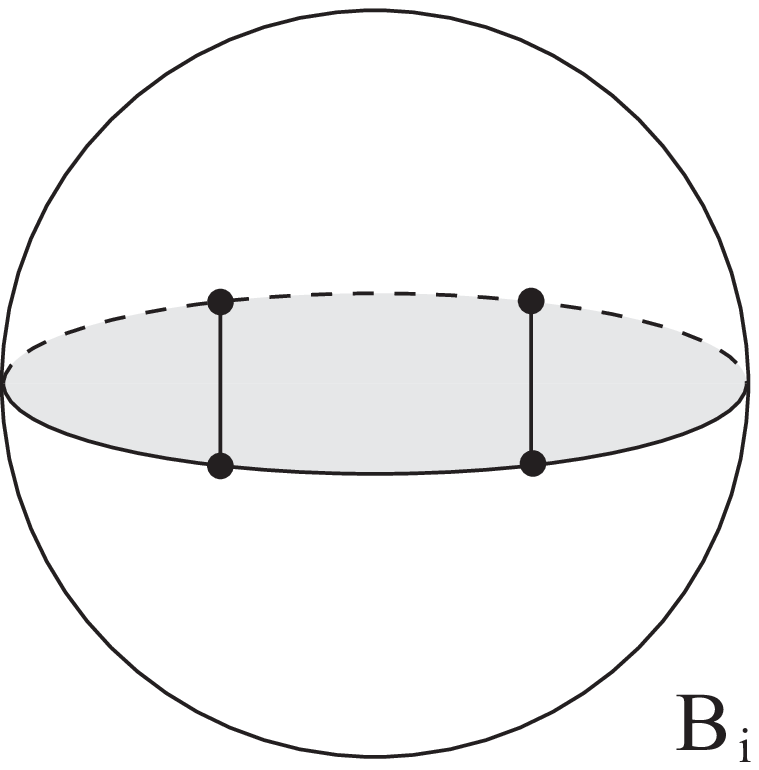}\\
	Type (A) & Type (B)\\
	\end{tabular}
	\end{center}
	\caption{Two types for $F\cap B_i$}
	\label{type}
\end{figure}

\begin{claim}\label{Type A}
For Type (A), the boundary slope of $F\cap B_i$ is equal to the rational tangle slope of $(B_i,T_i)$.
\end{claim}

\begin{proof}
This claim follows the definition of the slope of rational tangle.
\end{proof}

\begin{claim}\label{slope}
For Type (B), let $1/m$ be the slope of a rational tangle and $a/b$ be the boundary slope of $F\cap B_i$.
Then one of the following conditions holds.
\begin{itemize}
\item [(I)] $am=b-1$.
\item [(II)] $a>1$ and $am=b+1$.
\end{itemize}
\end{claim}

\begin{proof}
There are two possibility of the arrangement of $T_i$ in $F\cap B_i$.
Depending on them, we have two conditions (I) and (II).
See Figure \ref{cond}.
We remark that it is required $a>1$ in Condition (II).
\end{proof}

\begin{figure}[htbp]
	\begin{center}
	\begin{tabular}{ccc}
	\includegraphics[trim=0mm 0mm 0mm 0mm, width=.2\linewidth]{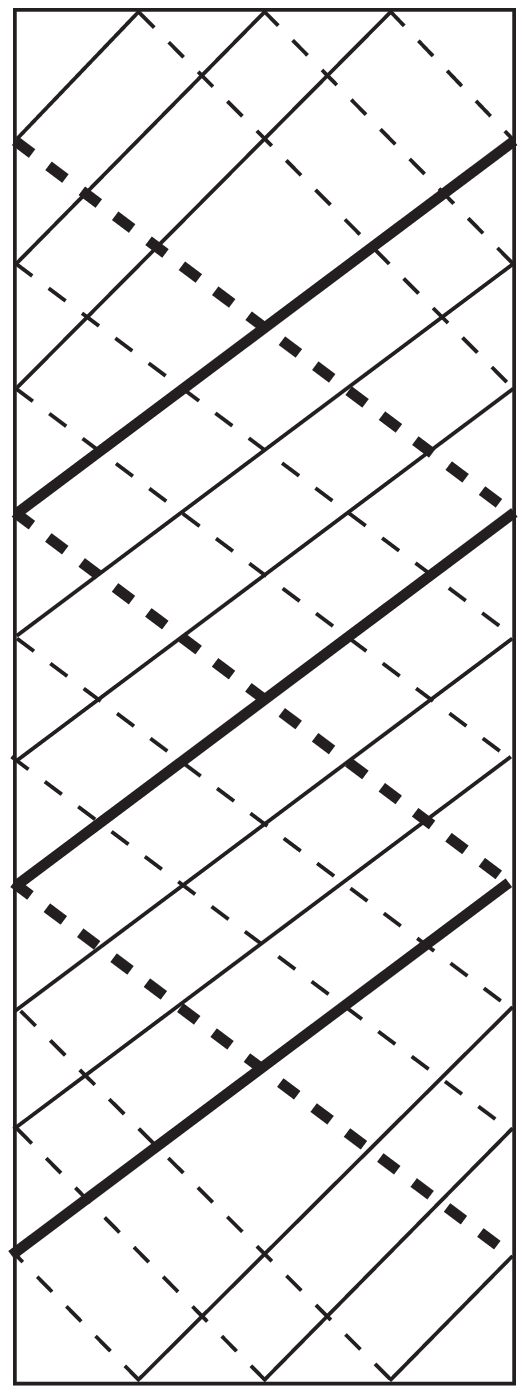}&
	\includegraphics[trim=0mm 0mm 0mm 0mm, width=.2\linewidth]{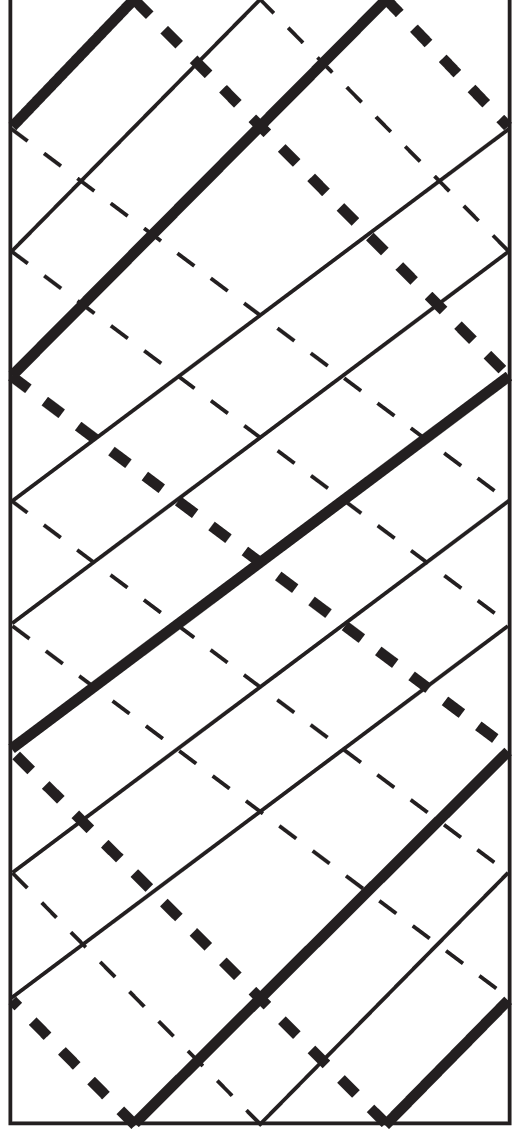}&
	\includegraphics[trim=0mm 0mm 0mm 0mm, width=.1\linewidth]{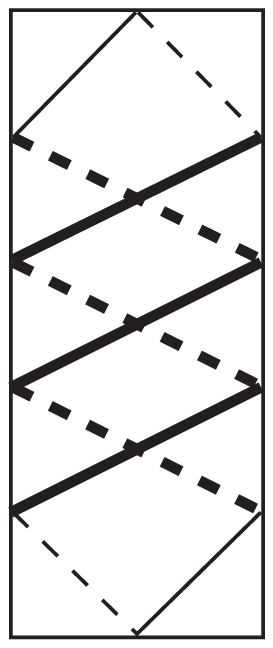}\\
	Condition (I) : slope $3/10$ & Condition (II) : slope $3/8$ & Condition (I) : slope $1/4$\\
	\end{tabular}
	\end{center}
	\caption{Two conditions for the boundary slope with the tangle slope $1/3$}
	\label{cond}
\end{figure}

\begin{claim}
In Claim \ref{slope}, $a=1$.
\end{claim}

\begin{proof}
Since $F\cap B_i$ is a disk in $B_i$, for each region $R$ between two consecutive arcs of $(F\cap B_i)\cap D_i$, there exists a disk $\delta$ in $B_i$ such that $\delta\cap R$ is an arc connecting the two consecutive arcs and $\delta\cap (F\cap B_i)=\partial \delta - \text{int}(\delta\cap R)$.
If $a\ge 2$, then $|\partial \delta \cap T_i|\le 1$.

For the next tangle $(B_j,T_j)$ which has a common disk $D_i$, we have a similar disk $\delta'$ in $B_j$ as $\delta$, for {\em some} region between two consecutive arcs of $(F\cap B_i)\cap D_i$, such that $|\delta\cap T_j|=1$.

Then we have a disk $\Delta=\delta\cup\delta'$ such that $|\Delta\cap K|\le 2$, a contradiction.
\end{proof}

Hence, only (I) in Claim \ref{slope} occurs, and we have the following claim.

\begin{claim}\label{Type B}
For Type (B), the boundary slope of $F\cap B_i$ is equal to $1/(m+1)$ if the slope of a rational tangle $(B_i,T_i)$ is $1/m$.
\end{claim}

Since $F\cap \partial V$ consists of longitudes, the next equation holds.
\[
\displaystyle \frac{1}{p'}+\frac{1}{q'}+\frac{1}{r'}=0,\ \ \ \ \ \ (*)
\]
where $p'$, $q'$ and $r'$ denote the boundary slopes of $F\cap B_1$, $F\cap B_2$ and $F\cap B_3$ respectively.
Without loss of generality, we assume that $p'<0$ and $q',r'>0$.
We note that the common denominator in the equation (*) is equal to $|F\cap D_i|=p'|F\cap B_1|=q'|F\cap B_2|=r'|F\cap B_3|$.
In the following, we will give a condition for $p',q',r'$.

\begin{lemma}\label{equation}
Generally, suppose that
\[
\displaystyle -\frac{1}{a}+\frac{1}{b}+\frac{1}{c}=0\ (2\le a < b \le c).
\]
Then there exists a pair of coprime integers such that $k<l\le 2k$ and an integer $d>0$ such that 
\[
a=k(l-k)d,\ b=l(l-k)d,\ c=kld.
\]
\end{lemma}

\begin{proof}
Put $a=km$, $b=lm$, where $k<l$ are coprime integers.
Then
\[
\displaystyle -\frac{1}{km}+\frac{1}{lm}+\frac{1}{c}=0
\]
\[
\displaystyle -\frac{l-k}{klm}+\frac{1}{c}=0
\]
Since $l-k$ divides $m$, putting $d=m/(l-k)$,
\[
\displaystyle -\frac{1}{kld}+\frac{1}{c}=0
\]
Hence, $a=k(l-k)d$, $b=l(l-k)d$ and $c=kld$.
Moreover, since $b\le c$, we have $l\le 2k$.
\end{proof}

\begin{claim}\label{Type 1}
If $l=2k$, then there exists a positive integer $d$ such that $(p,q,r)=(p',q'-1,r'-1)=(-d,2d-1,2d-1)$.
Here, $F\cap B_1$ is of Type (A) with $|F\cap B_1|=2$, and $F\cap B_2$ and $F\cap B_3$ are of Type (B).
\end{claim}

\begin{proof}
If $l=2k$, then $l=2$, $k=1$ since $l$ and $k$ are coprime.
Then we have $a=d$, $b=2d$ and $c=2d$ in Lemma \ref{equation}.
Since $K$ is a knot, it follows that $q=2d-1$ and $r=2d-1$, thus $F\cap B_2$ and $F\cap B_3$ are of Type (B) by Claim \ref{Type B}.
Then the common denominator in the equation (*) must be $2d$, $F\cap B_1$ is of Type (A) with $|F\cap B_1|=2$.
\end{proof}

\begin{example}
A $(-2,3,3)$-pretzel knot is of this type in Claim \ref{Type 1}.
\end{example}

\begin{claim}\label{Type 2}
If $l<2k$, then there exists positive integers $k,d$ such that $(p,q,r)=(p',q',r'-1)=(-kd,(k+1)d, k(k+1)d-1)$.
Here, $F\cap B_1$ is of Type (A) with $|F\cap B_1|=k+1$, $F\cap B_2$ is of Type (A) with $|F\cap B_2|=k$, and $F\cap B_3$ is of Type (B).
\end{claim}

\begin{proof}
If $l-k$ is even, then both of $a$ and $b$ are also even.
Since $b<c$, both of $F\cap B_1$ and $F\cap B_2$ are of Type (A) to have a common denominator in the equation (*).
Thus by Claim \ref{Type A}, the slopes of rational tangles $(B_1,T_1)$ and $(B_2,T_2)$ coincide with those boundary slopes of $F\cap B_1$ and $F\cap B_2$.
Therefore, $p$ and $q$ are even, this shows that $K$ has more than two components.

Hence $l-k$ is odd, and one of $l$ and $k$ is even.
It follows that one of $a$ and $b$ is even and $c$ is even.
Since $a<b<c$ and $K$ has exactly one component, $F\cap B_3$ is of Type (B) by Claim \ref{Type B}.
Then,
\[
\displaystyle -\frac{1}{k(l-k)d}+\frac{1}{l(l-k)d}+\frac{1}{kld}=0
\]
\[
\displaystyle -\frac{l-k}{kl(l-k)d}+\frac{1}{kld}=0
\]
Since $|F\cap D_2|=kl(l-k)d=kld$, $l-k=1$ and $l=k+1$.
Hence $a=kd$, $b=(k+1)d$ and $c=k(k+1)d$, and $F\cap B_1$ is of Type (A) with $|F\cap B_1|=k+1$, $F\cap B_2$ is of Type (A) with $|F\cap B_2|=k$ and $F\cap B_3$ is of Type (B).
\end{proof}

\begin{example}
A $(-2,3,5)$-pretzel knot is of this type in Claim \ref{Type 2}, where $k=2$ and $d=1$.
\end{example}

\begin{claim}
In Claim \ref{Type 2}, $d=1$.
\end{claim}

\begin{proof}
Suppose $d>1$.
There are $k(k+1)d$-arcs of $F\cap D_1$.
Let $\alpha$ be the $k(k+1)$-th arc and $\alpha'$ be the $k(k+1)+1$-th arc, and $R$ be the region of $D_1$ between $\alpha$ and $\alpha'$.
Then there exists a disk $\delta$ in $B_1$ such that $\delta\cap (F\cap B_1)$ is an subarc $\beta$ of $\partial \delta$ intersecting $T_1$ in one point and $\delta\cap R$ is an arc $\partial \delta -\text{int}\beta$ connecting $\alpha$ and $\alpha'$.
There also exists a disk $\delta'$ in $B_2$ such that $\delta'\cap (F\cap B_2)$ is an subarc $\beta'$ of $\partial \delta'$ intersecting $T_2$ in one point and $\delta'\cap R$ is an arc $\partial \delta' -\text{int}\beta'$ connecting $\alpha$ and $\alpha'$.
Then a disk obtained from $\delta$ and $\delta'$ gives a compressing disk for $F$ whose boundary intersects $K$ in two points.
This shows that $r(F,K)\le 2$, a contradiction.
\end{proof}

Therefore, we have two Types for the triple $(p,q,r)$.

\begin{itemize}
\item [(1)] $(p,q,r)=(-d,2d-1,2d-1)$ for some integer $d\ge 2$.
\item [(2)] $(p,q,r)=(-k,(k+1), k(k+1)-1)$ for some integer $k\ge 2$.
\end{itemize}

\begin{claim}
In Type (1), $d=2$.
\end{claim}

\begin{proof}
If $d\ge 3$, then we can find a compressing disk $\delta$ for $F$ such that $|\partial D\cap K|=2$, a contradiction.
See Figure \ref{(-3,5,5)}.
\end{proof}

\begin{figure}[htbp]
	\begin{center}
	\includegraphics[trim=0mm 0mm 0mm 0mm, width=.6\linewidth]{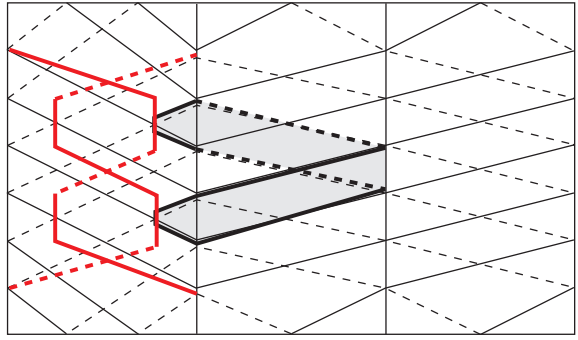}
	\end{center}
	\caption{a compressing disk $\delta$ for $F$ of Type (1) : $(-3,5,5)$-pretzel knot}
	\label{(-3,5,5)}
\end{figure}

\begin{claim}
In Type (2), $k=2$.
\end{claim}

\begin{proof}
If $k\ge 3$, then we can find a compressing disk $\delta$ for $F$ such that $|\partial D\cap K|=2$, a contradiction.
See Figure \ref{(-3,4,11)} for a $(-3,4,11)$-pretzel knot.
See Figure \ref{(-4,5,19)} for a $(-4,5,19)$-pretzel knot.
For $k\ge 5$, there exists a disk similar to Figure \ref{(-4,5,19)}.
\end{proof}

\begin{figure}[htbp]
	\begin{center}
	\includegraphics[trim=0mm 0mm 0mm 0mm, width=.6\linewidth]{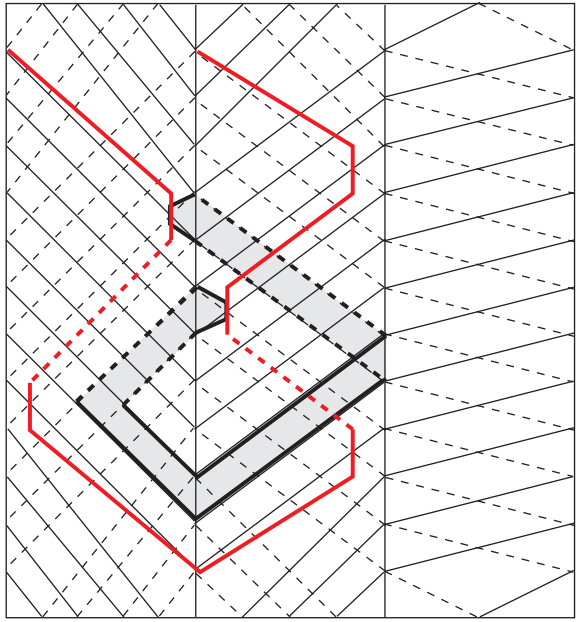}
	\end{center}
	\caption{a compressing disk $\delta$ for $F$ of Type (2) : $(-3,4,11)$-pretzel knot}
	\label{(-3,4,11)}
\end{figure}

\begin{figure}[htbp]
	\begin{center}
	\includegraphics[trim=0mm 0mm 0mm 0mm, width=.6\linewidth]{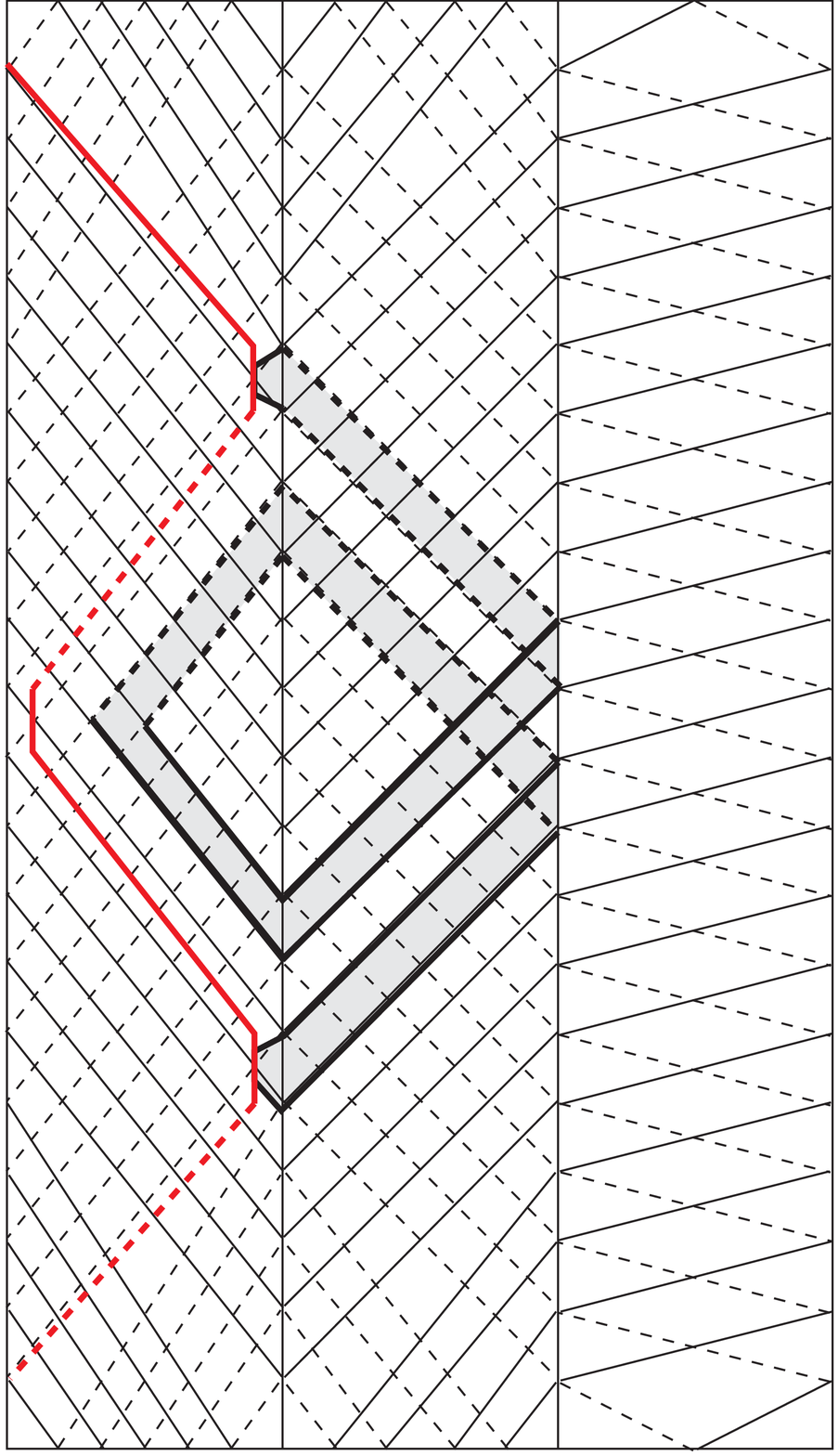}
	\end{center}
	\caption{a compressing disk $\delta$ for $F$ of Type (2) : $(-4,5,19)$-pretzel knot}
	\label{(-4,5,19)}
\end{figure}

Conversely, let $K$ be a $(-2,3,3)$-pretzel knot.
Then $K$ has a closed surface $F$ of Type (1), where $F\cap B_1$ is of Type (A) with $|F\cap B_1|=2$ and $F\cap B_2$ and $F\cap B_3$ are of Type (B).
See Figure \ref{(-2,3,3)}.
By the calculation of the Euler characteristic of $F$, 
\[
\chi(F)=(2+1+1)-(4+4)+4=0
\]
we know that $F$ is a torus, and there exist a meridian disks $D$ and $D'$ for each solid torus component obtained by splitting $S^3$ along $F$ such that $|\partial D\cap K|=3$ and $|\partial D'\cap K|=4$.
By checking that the algebraic intersection number of $\partial D\cap K$ and $\partial D'\cap K$ are equal to $3$ and $4$ respectively, it follows that $K$ is a $(3,4)$-torus knot and $F$ is an unknotted torus for $K$.

\begin{figure}[htbp]
	\begin{center}
	\includegraphics[trim=0mm 0mm 0mm 0mm, width=.6\linewidth]{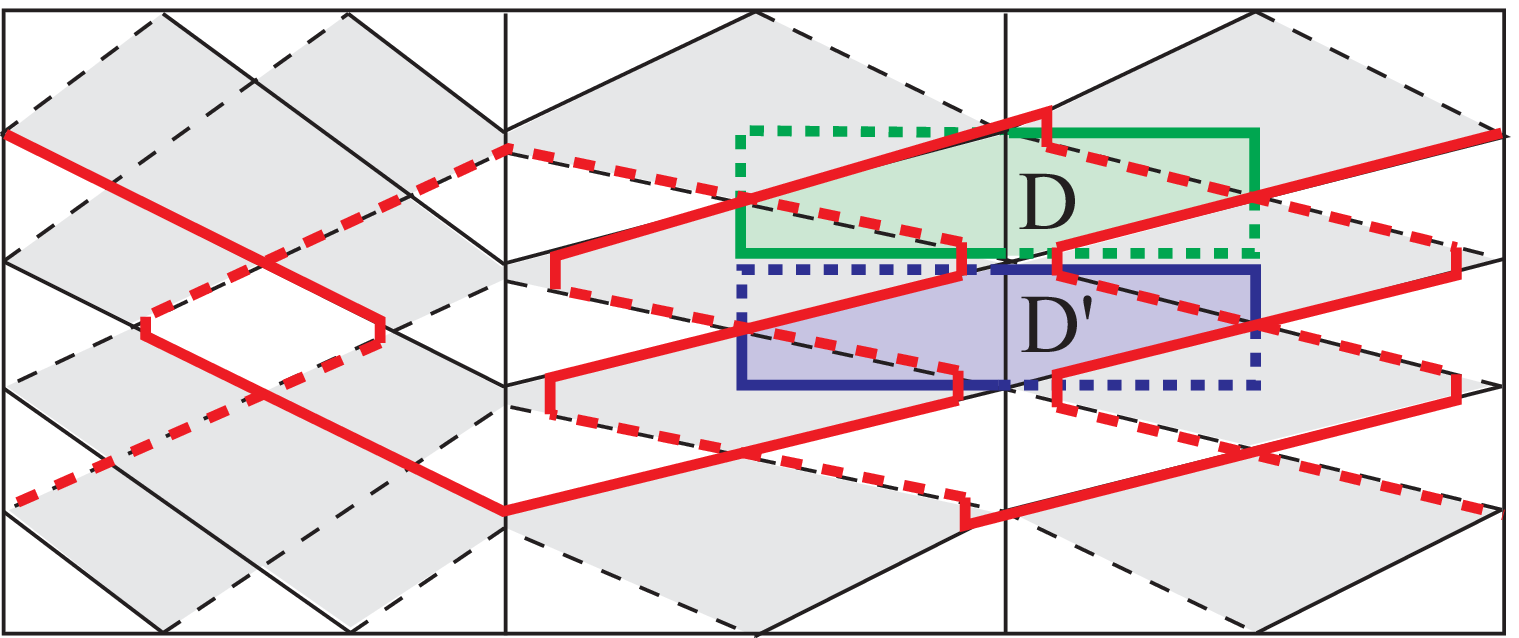}
	\end{center}
	\caption{a $(-2,3,3)$-pretzel knot and an unknotted torus for a $(3,4)$-torus knot}
	\label{(-2,3,3)}
\end{figure}

Next, let $K$ be a $(-2,3,5)$-pretzel knot.
Then $K$ has a closed surface $F$ of Type (2), where $F\cap B_1$ is of Type (A) with $|F\cap B_1|=3$, $F\cap B_2$ is of Type (A) with $|F\cap B_2|=2$ and $F\cap B_3$ are of Type (B).
See Figure \ref{(-2,3,5)}.
By the calculation of the Euler characteristic of $F$, 
\[
\chi(F)=(3+2+1)-(6+6)+6=0
\]
we know that $F$ is a torus, and there exist a meridian disks $D$ and $D'$ for each solid torus component obtained by splitting $S^3$ along $F$ such that $|\partial D\cap K|=3$ and $|\partial D'\cap K|=5$.
By checking that the algebraic intersection number of $\partial D\cap K$ and $\partial D'\cap K$ are equal to $3$ and $5$ respectively, it follows that $K$ is a $(3,5)$-torus knot and $F$ is an unknotted torus for $K$.
\begin{figure}[htbp]
	\begin{center}
	\includegraphics[trim=0mm 0mm 0mm 0mm, width=.6\linewidth]{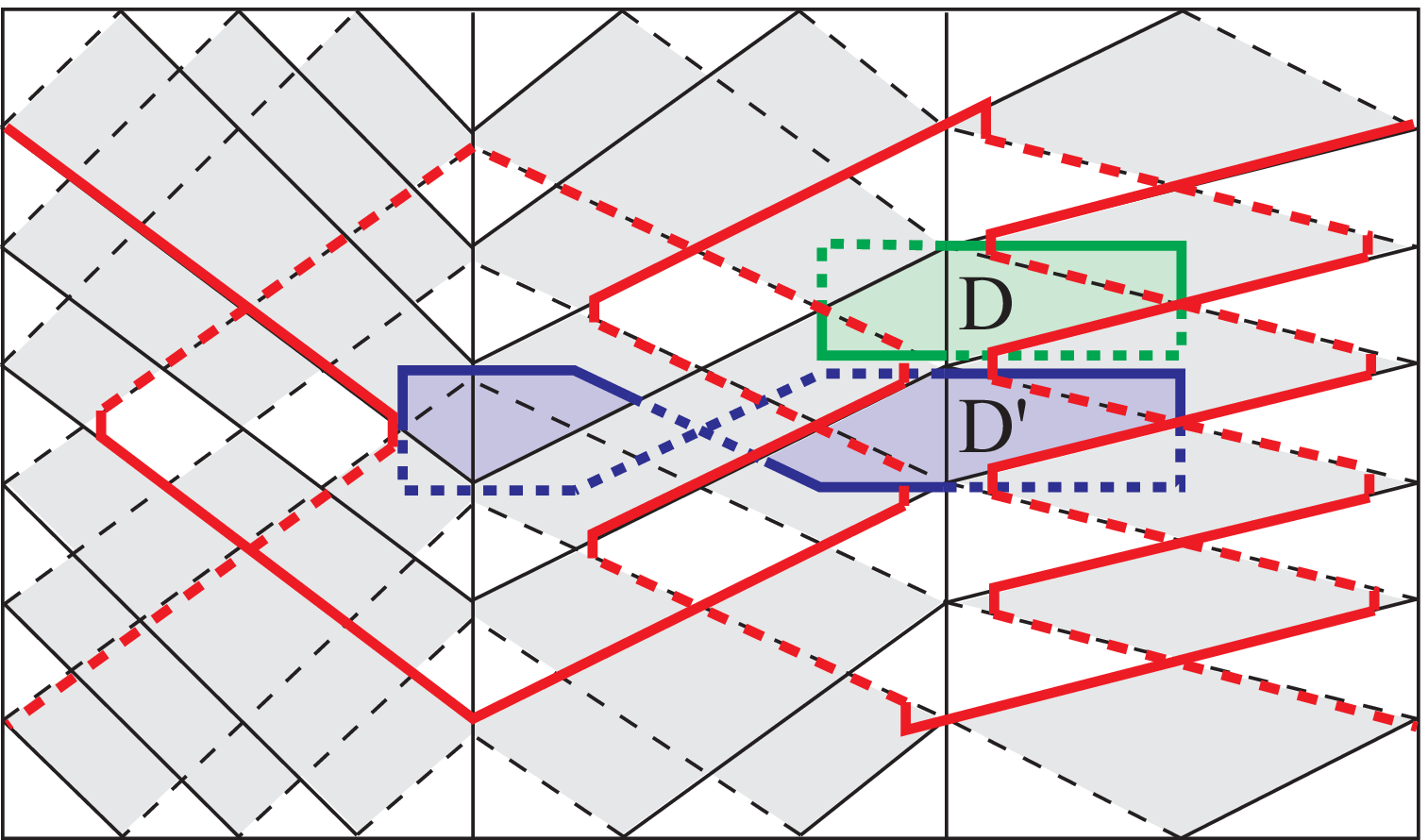}
	\end{center}
	\caption{a $(-2,3,5)$-pretzel knot and an unknotted torus for a $(3,5)$-torus knot}
	\label{(-2,3,5)}
\end{figure}
\end{proof}

\begin{proof} (of Theorem \ref{algebraic})
Let $K$ be a large algebraic knot.
Then $K$ admits an essential $2$-string tangle decomposition $(B,T)\cup(B',T')$ by a Conway sphere $S$ such that $(B,T)$ is obtained from two rational tangles $(B_1,T_1)$ and $(B_2,T_2)$ by gluing along a disk $D$.
Suppose that $r(K)=4$ and let $F$ be a closed surface containing $K$ so that $r(F,K)=4$.

We assume that $|F\cap S|$ is minimal up to isotopy of $F$.
Then $F\cap S$ consists of one loop intersecting $K$ in four points.
Moreover, we assume that $|(F\cap B)\cap D|$ is minimal up to isotopy of $F\cap B$.
Then $(F\cap B)\cap D$ consists of mutually parallel arcs whose two outermost arcs intersect $T$ in one point (c.f. Configuration (1)) or one arc intersecting $T$ in two points (c.f. Configuration (2)).
These are similar configurations as in Figure \ref{conf}.

In Configuration (1), let $\delta$ be the outermost disk in $D$ bounded by an outermost arc of $(F\cap B)\cap D$.
By cutting a disk in $S$ bounded by $F\cap S$ along an arc $\delta\cap S$, and pasting one of the resultant subdisks and $\delta$, we obtain a compressing disk $\delta'$ for $F$ such that $|\partial \delta'\cap K|\le 3$, a contradiction.

In Configuration (2), $F\cap \partial B_i$ consists of one loop.
By applying Lemma \ref{Morse} to $F\cap B_i$, there exists a compressing disk $\epsilon$ for $F\cap B_i$ such that $|\partial \epsilon\cap K|\le 2$  or $F\cap B_i$ is a disk.
In the formar case, we have a contradiction.
Hence $F\cap B_i$ is a disk and $F\cap B$ is also a disk.
However, this shows that a tangle $(B,T)$ is inessential, a contradiction.
\end{proof}

\bigskip


\bibliographystyle{amsplain}

\begin{thebibliography}{10}

\bibitem{AK} A. Kawauchi, {\em Classification of pretzel knots}, Kobe J. Math. {\bf 2} (1985) 11--22. 

\bibitem{MO1} M. Ozawa, {\em Non-triviality of generalized alternating knots}, J. Knot Theory and its Ramifications {\bf 15} (2006) 351--360.

\bibitem{MO} M. Ozawa, {\em Bridge position and the representativity of spatial graphs}, arXiv:0909.1162 (2009).


\end{thebibliography}

\end{document}